\documentclass[a4paper,reqno]{amsart}

\usepackage{amssymb}
\usepackage{verbatim}
\usepackage[enableskew]{youngtab}

\usepackage{centernot}

\usepackage{pgf}
\usepackage[normalem]{ulem}

\theoremstyle{plain}
\newtheorem{lemma}{Lemma}[section]
\newtheorem{theorem}[lemma]{Theorem}
\newtheorem{proposition}[lemma]{Proposition}
\newtheorem{corollary}[lemma]{Corollary}

\theoremstyle{definition}

\theoremstyle{remark}
\newtheorem{remark}[lemma]{Remark}

\newcommand{\mA}  {\mathbb A}
\newcommand{\mB}{\mathbb B} \newcommand{\mC}{\mathbb C}

\newcommand{\mN}{\mathbb N}  \newcommand{\mP}{\mathbb P}
\newcommand{\mQ}{\mathbb Q}

\newcommand{\mZ}{\mathbb Z}

\newcommand{\calA}  {\mathcal A}
 \newcommand{\calC}{\mathcal C} 
  
 \newcommand{\calI}{\mathcal I} 
 \newcommand{\calL}{\mathcal L} 
 \newcommand{\calO}{\mathcal O} 
 \newcommand{\calR}{\mathcal R} 
  
\newcommand{\calW}{\mathcal W} \newcommand{\calX}{\mathcal X}

\newcommand{\sfA}  {\mathsf A}

 \newcommand{\sfL}{\mathsf L} \newcommand{\sfM}{\mathsf M}
  
  \newcommand{\sfS}{\mathsf S}
\newcommand{\sfT}{\mathsf T}

\newcommand{\sfa}  {\mathsf a}
\newcommand{\sfb}{\mathsf b}

\newcommand{\sfk}{\mathsf k}  \newcommand{\sfm}{\mathsf m}
\newcommand{\sfn}{\mathsf n}  \newcommand{\sfp}{\mathsf p}

 \newcommand{\sfx}{\mathsf x}

\newcommand{\grS}{\Sigma}

\newcommand{\gra}{\alpha}        \newcommand{\grg}{\gamma}
  
 \newcommand{\grl}{\lambda}     \newcommand{\grs}{\sigma}
\newcommand{\grf}{\varphi}      

\newcommand{\grG}{\Gamma} \newcommand{\grD}{\Delta}
\newcommand{\grL}{\Lambda}

\newcommand{\mk}  {\Bbbk}

\newcommand{\SM}{\mathsf S \mathsf M}

\newcommand{\lra}      {\longrightarrow}
\newcommand{\vuoto}    {\varnothing}

\renewcommand{\geq}    {\geqslant}
\renewcommand{\leq}    {\leqslant}
\renewcommand{\setminus}    {\smallsetminus}
\newcommand{\ristretto}{\bigr|}

\DeclareMathOperator{\Hom}  {Hom}
\DeclareMathOperator{\End}  {End}

\DeclareMathOperator{\Pic}  {Pic}

\DeclareMathOperator{\Spec} {Spec}

\newcommand{\ol}{\overline}
\newcommand{\wt}{\widetilde}

            \newcommand{\st}       {\, : \,}
       
         \newcommand{\mand}     {\text{ and }}

\DeclareMathOperator{\Id}{Id}

\DeclareMathOperator{\supp}{supp}
\DeclareMathOperator{\Char}{char}

\DeclareMathOperator{\lex}{lex}

\newcommand*{\longhookrightarrow}{\ensuremath{\lhook\joinrel\relbar\joinrel\rightarrow}}

\newcommand{\Young}[1]{{\lower2pt\hbox{\young(#1)}}}
\newcommand{\YYoung}[1]{{\lower7pt\hbox{\young(#1)}}}
\newcommand{\ione}{i_1}
\newcommand{\itwo}{i_2}
\newcommand{\ienne}{i_N}
\newcommand{\jone}{j_1}
\newcommand{\jtwo}{j_2}
\newcommand{\jenne}{j_N}

\definecolor{revisionColor}{rgb}{0.0,0.0,0.0}
\newcommand{\rev}[1]{{\color{revisionColor}{#1}}}
\definecolor{rrevisionColor}{rgb}{0.0,0.0,0.0}
\newcommand{\rrev}[1]{{\color{rrevisionColor}{#1}}}
\definecolor{revColor}{rgb}{0.0,0.0,0}
\newcommand{\rv}[1]{{\color{revColor}{#1}}}
\newcommand{\revi}[1] {{\color{revColor}{#1}}}
\definecolor{revtwoColor}{rgb}{0.0,0.0,0.0}
\newcommand{\revtwo}[1]{{\color{revtwoColor}{#1}}}

\title{Standard monomial theory for wonderful varieties}

\author{P.~Bravi, R.~Chiriv\`\i, J.~Gandini and A.~Maffei}

\subjclass[2010]{Primary 14M27; Secondary 13F50, 20G05}

\keywords{Standard monomial theory, wonderful variety, degeneration, rational singularity}

\address{Dipartimento di Matematica ``Guido Castelnuovo''\\ Universit\`a di Roma ``La Sapienza''\\Piazzale Aldo Moro n.~5\\ 00185 Roma\\Italy}
\email{bravi@mat.uniroma1.i}

\address{Dipartimento di Matematica e Fisica ``Ennio De Giorgi''\\ Universit\`a del Salento\\ Via per Arnesano\\ 73047 Monteroni di Lecce (LE)\\ Italy}
\email{rocco.chirivi@unisalento.it}

\address{Scuola Normale Superiore di Pisa\\Piazza dei Cavalieri n.~7\\ 56127 Pisa (PI)\\Italy}
\email{jacopo.gandini@sns.it}

\address{Dipartimento di Matematica\\ Universit\`a di Pisa\\ Largo Bruno Pontecorvo 5\\ 56127 Pisa (PI)\\ Italy}
\email{maffei@dm.unipi.it}

\begin{document}

\begin{abstract}
A general setting for a \rv{standard monomial theory on a multiset} is introduced and applied to the Cox ring of a wonderful \rrev{variety}. This gives a degeneration result of the Cox ring to a multicone over a partial flag variety. Further, \rrev{we deduce that the Cox ring has rational singularities.}
\end{abstract}

\maketitle

\section{Introduction}

The first appearance of the idea of a standard monomial theory may be traced back to Hodge's study of 
Grassmannians in \cite{hodge_1943}, \cite{hodge_1994}. Then Doubilet, Rota and Stein found \rrev{a similar theory} 
for the coordinate \rv{ring} of the space of matrices in \cite{rota}. This was reproved and generalized to the space 
of symmetric and antisymmetric matrices by De Concini and Procesi in \cite{cdc_procesi}.

A systematic program for the development of a standard monomial theory for \rev{quotients} of reductive groups by 
parabolic subgroups was then started by Seshadri in \cite{seshadri} where the case of minuscule parabolics is considered. Further, in \cite{seshadri_lakshmibai} Seshadri and Lakshmibai noticed that the above \rrev{mentioned} results could be obtained as specializations of \rev{their} general theory.

This program was finally completed by Littelmann. Indeed, in \cite{littelmann_paths}, he found a combinatorial character formula for representations of symmetrizable Kac-Moody groups introducing the language of L-S paths. Moreover, he used \rrev{L-S paths} \rv{to construct a standard monomial theory for Schubert varieties of symmetrizable Kac-Moody groups in \cite{littelmann_contracting}}. This theory has been developed in the context of LS algebras over \rrev{posets} with bonds in \cite{chirivi_ls}, \cite{chirivi_thesis} and \cite{chirivi_multicone}.

We want now to briefly recall what a standard monomial theory is, the reader may see \cite{CLM} for further details about this general setting. Let $\mA$ be a finite subset of an algebra $A$ and suppose we are given a transitive antisymmetric binary relation $\longleftarrow$ on $\mA$. We define \rrev{a} \rev{formal} monomial $\sfa_1\sfa_2\cdots\sfa_N$ of elements of $\mA$ as standard if $\sfa_1\longleftarrow \sfa_2\longleftarrow\cdots\longleftarrow\sfa_N$. If the set of standard monomials is a basis of the algebra $A$ as a vector space then we say that $(\mA,\longleftarrow)$ is a standard monomial theory for $A$. Suppose, further, we have \rv{an order} $\leq$ on the monomials of elements of $\mA$. By the previous assumption, we may write any non-standard monomial $\sfm'$ as a linear combination of standard monomials. If in such an expression only standard monomials $\sfm$ with $\sfm'\rv{\leq}\sfm$ appear, then we say that we have a straightening relation for $\sfm'$. If we have a straightening relation for each non-standard monomial, then we say that $(\mA,\longleftarrow,\rv{\leq})$ is a standard monomial theory with straightening relations.

Given a simply connected semisimple algebraic group $G$ over an algebraically closed field $\mk$ of characteristic $0$, a Borel subgroup $B\subset G$ and a maximal torus $T\subset B$, let $\Lambda^+\subset\Lambda$ be the monoid of dominant weights and the lattice of weights, respectively. For a dominant weight $\lambda$, let $V_\lambda$ be the irreducible $G$--module of highest weight $\lambda$. Let $B\subset P\subset G$ be a parabolic subgroup of $G$ \rv{contained in the stabilizer of} the line generated by a highest weight vector in $V_\lambda$. Moreover, we denote by $\mB_\lambda$ the set of L-S paths of shape $\lambda$ \rv{(see Section \ref{sec:SMTMulticoneFlag} for \rv{details}).}

Littelmann's construction \rev{provides} a basis $\mA_\lambda=\{\rrev{\sfp_\pi}\,|\,\pi\in\mB_\lambda\}$, indexed by L-S paths, for the module $\Gamma(G/P,\calL_\lambda)\simeq V_\lambda^*$, where $\calL_\lambda$ is the line bundle over $G/P$ associated \rrev{with} $\lambda$. \rrev{The ring of sections $A_\lambda=\bigoplus_{n\geq0}\Gamma(G/P,\calL_{n\lambda})$ is generated in degree one and it is the coordinate ring of the cone over the closed embedding $G/P\longhookrightarrow\mP(V_\lambda)$ induced by $\calL_\lambda$.} On the basis $\mA_\lambda$ one may define a relation $\longleftarrow$ and \rv{an order} $\leq$ such that $(\mA_\lambda,\longleftarrow,\rv{\leq})$ is a standard monomial theory with straightening relations for \rev{$A_\lambda$}. 

In \cite{CM_ring}, the second and fourth named authors  adapted Littelmann's basis to the Cox ring (see below) of complete symmetric varieties;  this class of varieties has been introduced by \rrev{De Concini} and \rrev{Procesi} in \cite{CP}. As a \rrev{consequence}, they proved the degeneration of the Cox ring to the coordinate ring of a suitable multicone over a flag variety. This degeneration allowed a new proof of the rational singularity property for the Cox ring of complete symmetric varieties.

The purpose of the present paper is a further extension of these results to the Cox ring of wonderful varieties. As a first step, we \rrev{take the opportunity} to introduce a general setting for a \rv{standard monomial theory on a multiset} modelled on the above recalled one. This setting may be briefly summarized as follows, see Section \ref{sec:multiGradedSMT} below for details.

Let $\mA\doteq\mA_1\rev{\sqcup}\mA_2\rev{\sqcup}\cdots\rev{\sqcup}\mA_n$ be the union of \rev{disjoint} finite subsets of an algebra $A$. Suppose we have a binary relation $\longleftarrow$ on $\mA$ such that $\longleftarrow$ restricted to $\mA_i$ is transitive and antisymmetric for all $i=1,2,\ldots,n$ and, further, suppose we have bijective maps $\phi_{i,j}$\rev{, called swaps,} from the set of comparable pairs $\sfa\longleftarrow\sfb$ of $\mA_i\times\mA_j$ to the set of comparable pairs $\rev{\sfb'}\longleftarrow\rev{\sfa'}$ of $\mA_j\times\mA_i$ satisfying \rv{$\phi_{i,i} = \Id$ and $\phi_{i,j} \phi_{j,i} = \Id$}. We define a formal monomial $\sfa_1\sfa_2\cdots\sfa_N$ as weakly standard if $\sfa_1\longleftarrow\sfa_2\longleftarrow\cdots\longleftarrow\sfa_N$, and we say it is standard if all monomials obtained by \rrev{repeatedly} swapping \rrev{adjacent pairs in all possible ways, via the $\phi_{i,j}$'s,} are weakly standard.
If the set of standard monomials is a basis for $A$ as a vector space,
we say that $(\mA,\longleftarrow,\phi_{i,j})$ is a \rv{standard monomial theory on the multiset $\mA$ for the algebra $A$}. As above we introduce also \rrev{the notions of} \rv{order for monomials} \rrev{and} straightening relations for non-standard monomials.

\rev{We prove that the kernel of the natural map from the \rrev{symmetric} algebra over $\mA$ to $A$ is generated by the straightening relations of minimally non-standard monomials, that is by the straightening relations of those non-standard monomials which are not a product of non-standard monomials of smaller degree. In particular, \rrev{we show that} if any weakly standard monomial is standard then such kernel is generated in degree two.}

\rev{We \rrev{also show} how, given a valuation map for monomials that is compatible \rrev{with} the \rv{order $\leq$}, one may construct a flat degeneration of $A$.}

As a motivating example for this setting one may see the \rv{standard monomial theory on a multiset} for the multicone over a flag variety constructed by the second named author in \cite{chirivi_multicone}. This is described in details in Section \ref{sec:SMTMulticoneFlag}. \rv{In the same Section the multicone associated to fundamental weights for $\sfS\sfL$ and a multicone for $\sfS\sfL_2\times\sfS\sfL_2$ are studied in great detail.}

Now we recall \rrev{which is} the type of varieties we are interested in. A $G$--variety $X$ is wonderful of rank $r$ if it satisfies the following conditions:
\begin{itemize}
\item[--] $X$ is smooth and projective;
\item[--] $X$ possesses an open orbit whose complement is a union of $r$ smooth prime divisors, called the boundary divisors, with non-empty transversal intersections;
\item[--] any orbit closure in $X$ equals the intersection of the prime divisors which contain it.
\end{itemize}

Examples of wonderful varieties are the flag varieties, which are the wonderful varieties of rank zero, and the complete symmetric varieties. \rv{Wonderful varieties are instances of spherical varieties (\cite{luna_spherical}): normal \rv{$G$--varieties} with a dense orbit under the action of a Borel subgroup of $G$}. See \cite{BL} for a general introduction to wonderful varieties.

%
%

%
%

If $X$ is a wonderful $G$--variety, then the Picard group $\Pic(X)$ is freely generated by the classes of the $B$--stable prime divisors of $X$ which are not $G$--stable \rv{(see, for example, Section~2.2 in \cite{brion_cox})}. These divisors are called the colors of $X$. Since $X$ contains an open $B$--orbit, the colors form a finite set $\grD$, so that $\Pic(X)$ is a free lattice of finite rank. \rev{Given $D \in \mZ\grD$, we denote by $\calL_D$ the corresponding line bundle.}

The direct sum
$$
C(X) \quad \doteq \bigoplus _{\calL \in \Pic(X)} \Gamma(X,\calL),
$$
\rev{has a ring structure (see \rrev{Section~\ref{sec:SMTCox}} below) and it} is called the Cox ring of $X$.

%
%

%
%

Denote by $\grs_1, \ldots, \grs_r$ the boundary divisors of $X$, and let $s_i$ be \rv{a} \rev{section of $\rev{\calO(\grs_i)}$ defining $\sigma_i$}, for $i=1,\ldots,r$. As an algebra $C(X)$ is generated by the sections of the line bundles $\calL_D\rev{=\calO(D)}$ with $D \in \grD$ together with the sections $s_1, \ldots, s_r$.

%
%

By definition, $X$ contains a unique closed $G$--orbit $Y \simeq G/P$ \rev{for a suitable parabolic subgroup $P$}, and given $D \in \mN \grD$ we denote by $\grl_D$ the highest weight of the dual of the simple $G$--module $\grG(Y,\calL_D\ristretto_Y)$, so that $\calL_D\ristretto_Y \simeq \calL_{\grl_D}$ corresponds to the equivariant line bundle on $G/P$ associated to the dominant weight $\grl_D$. By taking into account the \rv{decomposition} of $\Gamma(X,\calL_D)$ as a $G$--module \rv{(see Proposition \ref{prop: decomposizione sezioni})}, we lift Littelmann's basis of $\grG(Y,\calL_{\grl_D})$ to $X$, and we take as algebra generators for $C(X)$ this set of lifts together with the sections $s_1, \ldots, s_r$.

\rev{Consider the coordinate ring $C(Y)=\oplus_{E\in\mN\Delta}\Gamma(Y,\rrev{\calL_E\ristretto_Y})$ of the multicone over the flag variety $Y$ associated to the dominant weights $\lambda_D$, with $D\in\Delta$.}
\rev{In Section \ref{sec:SMTCox} we construct a \rv{standard monomial theory on a multiset} for $C(X)$ by extending, in a natural way, \rv{that} of $C(Y)$.} \rv{Further an explicit example of our construction is given.}

%
%

As a consequence of our standard monomial theory, we obtain a flat deformation which degenerates $C(X)$ to the product $\mk[s_1, \ldots, s_r] \otimes \rev{C(Y)}$. Since multicones over flag varieties have rational singularities by \cite{kempf} and, since \rrev{the property of having} rational singularities \rrev{is} stable under deformation by \cite{elkik}, it follows that the Cox ring $C(X)$ has rational singularities as well. \rev{From this it follows at once that, given $D\in\mZ\Delta$, also the ring $C_D(X)=\oplus_{n\geq0}\Gamma(X,\calL_{nD})$ has rational singularities.} Both $C(X)$ and $C_D(X)$, for any $D\in\mZ\Delta$, can be seen as coordinate rings of affine spherical varieties, see Section~3.1 in \cite{brion_cox}, and the fact that affine spherical varieties have rational singularities was already known, see \cite{popov} and \cite{alexeev-brion}.

\section{\rv{Standard monomial theory on a multiset}}\label{sec:multiGradedSMT}

In this section, as a first step, we introduce the notion of a \rv{standard monomial theory on a multiset}. This requires some technical machinery which we express in a very abstract setting. In the next section we see the application to the multicone over a flag variety \rev{and in Section \ref{sec:SMTCox} that to the Cox ring of a wonderful variety.}

For further details the reader may see the various referenced papers as suggested below. In particular, the \rv{standard monomial theory} we are going to introduce is modelled \rrev{on} the general definition of a standard monomial theory given in \cite{CM_modelVarieties}; see also \cite{chirivi_thesis} and \cite{chirivi_multicone} where such kind of standard monomial theory is developed in the context of \rrev{posets} with bonds.

We begin with a field $\mk$ and a commutative $\mk$--algebra $A$. Let $\mA_1,\mA_2,\ldots,\mA_n$ be disjoint finite subsets of $A$ \rv{and} let $\mA\doteq\mA_1\rev{\sqcup}\mA_2\rev{\sqcup}\cdots\rev{\sqcup}\mA_n$\rv{, we call $\mA$ a \emph{multiset}}. 
\revi{By \emph{formal monomial} we mean a monomial in the elements of $\mA$ in the free associative algebra generated by $\mA$.} 
\rv{Define} the \emph{shape} of $\sfa\in\mA_i$ as the index $i$ \rv{and extend} the notion of shape to formal monomials $\sfm\doteq \sfa_1 \sfa_2\cdots \sfa_N$ of elements of $\mA$ by declaring that the shape of $\sfm$ is $(i_1,i_2,\ldots,i_N)$ if $\sfa_h$ has shape $i_h$ for $h=1,2,\ldots,N$.

Suppose we have a binary relation $\longleftarrow$ on $\mA$ that is antisymmetric \rev{and transitive} when restricted to $\mA_i$ for all $i$. We say that a \rv{formal} monomial $\sfa_1 \sfa_2 \cdots \sfa_N$ of elements of $\mA$ is \emph{weakly standard} if $\sfa_1 \longleftarrow \sfa_2 \longleftarrow \cdots \longleftarrow \sfa_N$. Given $i,j$, let $\phi_{i,j}$ be a map from the set of weakly standard \rv{formal} monomials of shape $(i,j)$ to the set of weakly standard \rv{formal} monomials of shape $(j,i)$. \rv{We assume that these maps verify} $\phi_{i,i} = \Id$ and $\phi_{i,j} \phi_{j,i} = \Id$, \rv{and} we call them \rrev{\emph{swap maps}}.

Now let $\sfm = \sfa_1 \sfa_2 \cdots \sfa_N$ be a weakly standard \rv{formal} monomial. Since any pair $\sfa_{j} \sfa_{j+1}$ is weakly standard, we may swap \rrev{it} and obtain a new monomial
$$
	\sfm' \doteq \sfa_1\cdots \sfa_{j-1} \sfa'_{j} \sfa'_{j+1} \sfa_{j+2} \cdots \sfa_N,
$$
where $\sfa'_j \sfa'_{j+1}$ is the swap of $\sfa_{j} \sfa_{j+1}$. If also $\sfm'$ is weakly standard, then we may \rrev{apply another} swap, etc. If the shape of $\sfm$ is a non-decreasing sequence and if all monomials obtained from $\sfm$ by swaps are weakly standard, then we say that $\sfm$ is a \emph{standard \rv{formal} monomial} (notice that the number of swaps of $\sfm$ is surely finite since $\mA$ is a finite set).

\rv{We say that a (commutative) monomial in the symmetric algebra $\mathsf{S}(\mA)$ is weakly standard, respectively, standard if it is the image of a weakly standard, respectively, standard formal monomial of elements of $\mA$ via the natural map from formal monomials to monomials in $\mathsf{S}(\mA)$.}

\rv{
We say that the above datum $(\mA,\phi_{i,j},\longleftarrow)$ is a \emph{\rv{standard monomial theory on the multiset $\mA$}} for $A$ if: the images of the standard monomials of $\mathsf{S}(\mA)$ in $A$ via the natural map $\mathsf{S}(\mA)\longrightarrow A$ are all \rv{distinct} and, moreover, the set of these standard monomials is a basis of $A$ as a $\mk$--vector space.
}

A standard monomial theory usually has another feature, the straightening relations. \rv{We express them by introducing an order $\leq$ on (commutative) monomials in $\mA$ as elements of $A$ with the following properties:}
\begin{itemize}
	\item[i)] if $\sfm,\,\sfm',\,\sfm''$ are monomials in $\mA$ and if $\sfm'\leq \sfm''$, then $\sfm \sfm' \leq \sfm \sfm''$,
	\item[ii)] \rv{for every monomial $\sfm$ the set of monomials $\sfm'$ such that $\sfm\leq\sfm'$ is a finite set.}
\end{itemize}
Since the standard monomials are a $\mk$--basis of $A$, for every non-standard monomial $\sfm'$ of elements of $\mA$ we have a relation
$$
\sfm' = \sum_\sfm a_\sfm \sfm
$$
\rv{in $A$,} expressing $\sfm'$ as a linear combination of the standard monomials $\sfm$.
\rv{If} we have $\sfm' \rv{\leq} \sfm$ whenever $a_\sfm \neq 0$ we say that the above relation is a \emph{straightening relation} for $\sfm'$. If, further, we have straightening relations for all non-standard monomials then we say that $(\mA,\phi_{i,j},\longleftarrow,\rv{\leq})$ is a \rv{standard monomial theory on a multiset} \emph{with straightening relations} for $A$.

\rev{A \rv{formal} non-standard monomial $\sfm\rv{=\sfa_1\sfa_2\cdots\sfa_N}$ is 
\emph{minimally} non-standard if \rv{for all proper \rv{subsequences} $1\leq i_1<i_2<\cdots<i_k\leq N$ of indexes, 
the monomial $\sfa_{i_1}\sfa_{i_2}\cdots\sfa_{i_k}$ is a formal standard monomial.} \revi{We say that a monomial in $\sfS(\sfA)$ is minimally non-standard if it is the image of a minimally non-standard formal monomial.}

Notice that the straightening relations generate the kernel $\calR$ of the map $\sfS(\mA)\longrightarrow A$, i.e. the ideal of relations in the generators $\mA$ for $A$. But fewer relations can suffice as we see in the following theorem.
\begin{theorem}\label{thm:minimallyNonStandard} If $(\mA,\phi_{i,j},\longleftarrow,\rv{\leq})$ is a standard monomial theory with straightening relations for $A$, then $\calR$ is generated by the straightening relations of the minimally non-standard monomials.
\end{theorem}
\begin{proof} Let $\calI$ be the ideal of $\sfS(\mA)$ generated by the straightening relations for minimally non-standard $\sfm$.
\rrev{Since by definition $\calI\subseteq\calR$ we have a surjective map $\sfS(\mA)/\calI\longrightarrow\sfS(\mA)/\calR$. Moreover the set of standard monomials are a $\mk$--basis for $\sfS(\mA)/\calR\simeq A$, so if we show that the images of standard monomials \revi{generate} $\sfS(\mA)/\calI$ \revi{as a vector space} 
we have $\calI=\calR$.

\rv{In order to prove that in $\sfS(\mA)/\calI$ any monomial $\sfm$ is a $\mk$--linear combination of standard monomials we use induction on the \rv{order $\leq$}. In particular, if $\sfm$ is $\rv{\leq}$--maximal then it is standard by the order requirement in the straightening relations.

So now suppose that $\sfm$ is non-standard and let $\sfm_1,\,\sfm_2$ be monomials such that $\sfm=\sfm_1\sfm_2$ with $\sfm_1$ minimally non-standard. In $\sfS(\mA)/\calI$ we have $\sfm_1=\sum_\sfn a_\sfn\sfn$ where the sum runs over the standard monomials $\sfn$ with $\sfm_1<\sfn$. 
}
%
%
%
\rv{Hence}
$$
\sfm=\sfm_1\sfm_2\equiv\sum a_{\sfn}\sfn\sfm_2\,\,\,(\textrm{mod}\,\,\calI)
$$
and $\sfm<\sfn\sfm_2$ for all $\sfn$. Using the inductive hypothesis on $\rv{\leq}$, all $\sfn\sfm_2$'s are sums of standard monomials, hence also $\sfm$ is sum of standard monomials in $\sfS(\mA)/\calI$.
}
\end{proof}

As a corollary we have the following result.

\begin{corollary}\label{cor:generatedDegreeTwo} If all weakly standard monomials are standard \rv{and the ideal $\calR$ of relations is homogeneous for the total degree of monomials, then $\calR$} is generated in degree two.
\end{corollary}
\begin{proof}\rv{First of all we show that for a minimally non-standard monomial $\sfm=\sfa_1\sfa_2\cdots\sfa_N$ we have $N=2$. Indeed suppose $N\geq3$, then $\sfa_i\sfa_{i+1}$ is standard for all $i=1,2,\ldots,N-1$. So $\sfa_1\longleftarrow\sfa_2\longleftarrow\cdots\longleftarrow\sfa_N$ and $\sfm$ is weakly standard, hence it is standard.

We know that $\calR$ is generated by the straightening relations of the minimally non-standard monomials by the previous theorem. We have just seen that minimally non-standard monomials have degree $2$, so $\calR$ is generated by the straightening relations of non-standard monomials of total degree $2$. Let $\sfm$ be such a monomial, let
$$
R\doteq\sfm-\sum_\sfn a_\sfn\sfn
$$
be its straightening relations and finally let
$$
R_2\doteq\sfm-\sum_\sfn a_\sfn\sfn
$$
where the sum is over all standard monomials $\sfn$ of total degree $2$. The ideal $\calR$ is homogeneous, hence $R_2\in\calR$; but then $R'\doteq R_2 - R\in\calR$ and $R'$ is a sum of standard monomials. This is clearly possible only if $R'=0$ since the standard monomials are linearly independent in $A\simeq\sfS(\mA)/\calR$. This shows that $R=R_2$ is an homogeneous relation of total degree $2$ and completes the proof.
}
\end{proof}
}

\rev{
In the last part of this section we see how a degeneration for $A$ may be constructed using the \rrev{straightening} relations (see also \cite{chirivi_thesis} for further details about this kind of degeneration in the language of LS algebras). Suppose we have a \emph{valuation}: a map $\delta:\mA\longrightarrow\mN$ such that, when extended to monomials \rrev{by} $\delta(\sfm\sfn)=\delta(\sfm)+\delta(\sfn)$ for all $\sfm,\,\sfn$, we have $\delta(\sfm)\leq\delta(\sfn)$ if $\sfm\rv{\leq}\sfn$.

For an integer $n$, let $K_n$ be the ideal of $A$ generated by those monomials $\sfm$ such that $\delta(\sfm)\geq n$ and consider the Rees algebra
$$
\calA \doteq \cdots\oplus At^2\oplus At\oplus A\oplus K_1t^{-1}\oplus K_2t^{-2}\oplus\cdots
$$
as a subalgebra of $\mk[t,t^{-1}]\otimes A$. This algebra is a torsion-free $\mk[t]$--module, hence it is a flat $\mk[t]$--algebra. \rrev{For $a\in\mk$ let $\calA_a\doteq\calA/(t-a)$ be the fiber over $a$. Notice that we have an action of $\mk^*$ on $\calA$ given by $\lambda\cdot t = \lambda t$, for all $\lambda\in\mk^*$; hence isomorphisms $\calA_a\longrightarrow\calA_{\lambda^{-1}a}$ between the fibers. In particular all \emph{generic fibers}, i.e.~$\calA_a$ with $a\neq0$, are isomorphic to $\calA_1\simeq A$. On the other hand} the \emph{special fiber} $\calA_0=\calA/(t)$ is isomorphic to the associated graded algebra
$$
A/K_1\oplus K_1/K_2\oplus K_2/K_3\oplus\cdots
$$
\rrev{We may now state our deformation result.}
\begin{theorem}\label{thm:degeneration} \rv{Let $(A,\varphi_{i,j},\longleftarrow,\rv{\leq})$ be a standard monomial theory with straightening relations for the ring $A$ and let $\delta$ be a valuation as above. Then} there exists a flat $\mk^*$--equivariant degeneration of $A$ to $\calA_0$ whose all generic fibers are isomorphic to $A$ while the special fiber $\calA_0$ is isomorphic to the quotient of the \rrev{symmetric} algebra $\sfS(\mA)$ by the ideal generated by \rrev{the relations}
$$
\sfm'\quad - \sum_{\delta(\sfm)=\delta(\sfm')}a_\sfm\sfm
$$
\rrev{where $\sfm'$ is a minimally non-standard monomial and $\sfm'\, - \sum_{\sfm}a_\sfm\sfm$ is its straightening relation.}
\end{theorem}
\begin{proof} We have only to prove the last part about $\calA_0$. \rrev{Consider the symmetric algebra $\sfT\doteq\sfS(\mA,t)$ with indeterminates the set of generators $\mA$ and the parameter $t$. Let $B$ be the quotient of $\sfT$ by the ideal generated by the modified straightening relations
$$
\sfm'-\sum_{\sfm}a_\sfm\sfm t^{\delta(\sfm)-\delta(\sfm')}
$$
for all $\sfm'$ minimally non-standard. We may define a map $\sfT\longrightarrow\mk[t,t^{-1}]\otimes A$ by $\mA\ni\sfa\longmapsto\sfa t^{-\delta(\sfa)}$ and by $t\longmapsto t$. It is clear that this map is well defined also on $B$. Its image is $\calA$ by definition of this last algebra. Moreover it is an injective map since $\calA$ has a \rv{standard monomial theory on a multiset} defined in terms of that of $A$ with any monomial $\sfm$ replaced by $\sfm t^{-\delta(\sfm)}$.

It is now clear that $\calA_0\simeq B\ristretto_{t=0}$ is as claimed in the statement of the theorem.
}
\end{proof}
}

\section{Standard monomial theory for multicones over flag varieties}\label{sec:SMTMulticoneFlag}


In this section we apply the abstract construction of the previous section to the multicone over a flag variety; this is the motivating example for the above general setting of a \rv{standard monomial theory on a multiset}.

Let $G$ be a simply connected semisimple algebraic group over an algebraically closed field $\mk$ of characteristic $0$ and let $T\subset B\subset G$ be a maximal torus and a Borel subgroup of $G$, respectively. Denote by $W$ the Weyl group and by $\Lambda\supset\Lambda^+$ the lattice of integral weights and the monoid of dominant weights associated to the choice of $T$ and $B$. For a dominant weight $\lambda$ denote by $W_\lambda \subseteq W$ its stabilizer and by $W^\lambda\subseteq W$ the set of minimal \rv{length} representatives of the cosets $W/W_\lambda$; denote, moreover, by $\leq$ the Bruhat order on $W$ and on $W^\lambda$.

Now let $P\supseteq B$ be a parabolic subgroup of $G$ stabilizing \rev{the line generated by a highest weight vector $v_\lambda$ in the irreducible $G$--module $V_\lambda$ of highest weight $\lambda$}. We have a natural map $G/P\ni gP\longmapsto [g\cdot v_\lambda]\in\mP(V_\lambda)$ \rev{from the flag variety $G/P$ to the projective space over $V_\lambda$}. We use this map to define the line bundle $\calL_\lambda$ on $G/P$ as the pull-back of $\calO(1)$ on $\mP(V_\lambda)$; we denote \rrev{the space of its sections by} $\Gamma(G/P,\calL_\lambda)$. Notice that, as $G$--\rev{modules}, we have $\Gamma(G/P,\calL_\lambda)\simeq V_\lambda^*$, the dual of $V_\lambda$. In the sequel we denote by $\lambda^*$ the unique dominant weight such that $V_{\lambda^*}\simeq V_\lambda^*$ as $G$--modules.

%
%

In \cite{littelmann_paths} Littelmann associated to a fixed piece-wise linear path $\pi:[0,1]_\mQ\longrightarrow \Lambda\otimes\mQ$ completely contained in the dominant Weyl chamber and ending in $\lambda\doteq\pi(1)\in\Lambda^+$, a set $\mathbb{B}_\pi$ of piece-wise linear paths. The set $\mathbb{B}_\pi$ gives the character of the irreducible module $V_\lambda$:
$$
\Char V_\lambda = \sum_{\eta\in\mathbb{B}_\pi}e^{\eta(1)}
$$
In particular, if we start with the path $\pi_\lambda:t\longmapsto t\lambda$ we obtain the set $\mathbb{B}_\lambda$ of \emph{L-S paths} \rev{of \emph{shape} $\lambda$}; they may be combinatorially described in the following way.

Given a pair $\tau < s_\beta\tau$ of adjacent elements in $W^\lambda$, where $s_\beta$ is the \rv{reflection} with respect to the \rv{positive} root $\beta$, we define \rv{the positive integer} $f_\lambda(\tau,s_\beta\tau)\doteq\langle\tau(\lambda),\check{\beta}\rangle$. Further, we extend $f_\lambda$ to \rv{comparable} \rev{pairs} $\sigma<\tau$ in $W^\lambda$ by choosing a chain $\sigma=\tau_1<\tau_2<\cdots<\tau_u=\tau$ of adjacent elements in $W^\lambda$ and defining $f_\lambda(\sigma,\tau) \doteq \gcd\{f_\lambda(\tau_1,\tau_2),\ldots,f_\lambda(\tau_{u-1},\tau_u)\}$; indeed such $\gcd$ is independent of the chain used to compute it (see \cite{raika}).

A pair $\eta\doteq(\tau_1,\tau_2,\ldots,\tau_r;a_0,a_1,\ldots,a_r)$, where $\tau_1<\tau_2<\cdots<\tau_r$ is a sequence of comparable elements of $W^\lambda$ and $0=a_0<a_1<\cdots<a_r=1$ are rational numbers, is an L-S path if the integral condition $a_if_\lambda(\tau_i,\tau_{i+1})\in\mN$ holds for all $i=1,2,\ldots,r-1$. The pair $\eta$ is identified with the path
$$
\pi_{(a_r-a_{r-1})\tau_r(\lambda)}*\pi_{(a_{r-1}-a_{r-2})\tau_{r-1}(\lambda)}*\cdots*\pi_{(a_1-a_0)\tau_1(\lambda)}
$$
where we denote concatenation of paths by $*$. The set $\supp\eta\doteq\{\tau_1,\tau_2,\ldots,\tau_r\}$ is called the \emph{support} of the path $\eta$.

Let $\calW$ be the set of words in the alphabet $W$ and denote by $N_\lambda$ the least common multiple of the image of $f_\lambda$; we define the \emph{word} $w(\eta)$ of the L-S path $\eta=(\tau_1,\tau_2,\ldots,\tau_r;a_0,a_1,\ldots,a_r)$ as $w(\eta)\doteq\tau_1^{N_\lambda(a_1-a_0)}\cdots\tau_r^{N_\lambda(a_r-a_{r-1})}$; this will be needed in the sequel \rev{to define} \rv{an order on monomials}.

The set $\mathbb{B}_\lambda$ not only describes the character of the irreducible $G$--module $V_\lambda$, but also, in \cite{littelmann_contracting}, Littelmann associates a section $\rev{\sfp_\pi}$ in $\Gamma(G/P,\calL_\lambda)$ to an L-S path $\rrev{\pi}\in\mathbb{B}_\lambda$. \rev{The set $\mA_\lambda\doteq\{\sfp_\pi\,|\,\pi\in\mB_\lambda\}$ of these sections, of shape $\lambda$,} may be used to construct a standard monomial theory as follows. For more details about the combinatorics of L-S paths and their application to the geometry of Schubert varieties one may see \cite{chirivi_ls}.

Given two dominant weights $\lambda,\mu$ \rev{we} lift the Bruhat order on $W^\lambda$ and $W^\mu$ to $W^\lambda\sqcup W^\mu$ by defining $W^\lambda\ni\sigma\leq\tau\in W^\mu$ if there exist $\sigma',\tau'\in W$ such that $\sigma'W_\lambda=\sigma W_\lambda$, $\tau'W_\mu=\tau W_\mu$ and $\sigma'\leq\tau'$ with respect \rv{to} the Bruhat order of $W$. For details we refer to \cite{chirivi_multicone}. Notice that this lift is still the Bruhat order if $\lambda$ and $\mu$ have the same stabilizer in $W$. We use this order to define a relation $\underset{\lambda,\mu}\longleftarrow$ on pairs in $\mB_\lambda\times\mB_\mu$ as follows: we set
$$
 \pi\underset{\lambda,\mu}\longleftarrow \eta \quad \text{ if }\pi\in\mB_\lambda,\,\eta\in\rev{\mB_\mu},\,\text{ and }\max\supp\pi\leq\min\supp\eta
$$
Notice that if $\lambda = \mu$, then $\underset{\lambda,\lambda}\longleftarrow$ is a transitive and antisymmetric relation.

Recall that the set of pairs $(\pi, \eta) \in \mB_\lambda \times \mB_\mu$ such that $\pi\underset{\lambda,\mu}\longleftarrow\eta$ is in natural bijection with the basis $\mB_{\pi_\lambda*\pi_\mu}$ as proved by Littelmann in \rv{Theorem~10.1 in} \cite{littelmann_plactic}. Further the two \rev{bases} $\mathbb{B}_{\pi_\lambda*\pi_\mu}$ and $\mathbb{B}_{\pi_\mu*\pi_\lambda}$ of the \rev{module} $V_{\lambda+\mu}$ are in bijection by a unique isomorphism of crystal graphs \rv{as follows at once by Theorem~6.3 in \cite{littelmann_paths}}. So we have the diagram
$$
\begin{array}{ccc}
\{\pi*\eta\,|\,\pi\underset{\lambda,\mu}\longleftarrow\eta\} & \longrightarrow & \mB_{\pi_\lambda*\pi_\mu}\\
 & & \downarrow\\
\{\eta'*\pi'\,|\,\eta'\underset{\mu,\lambda}\longleftarrow\pi'\} & \longleftarrow & \mB_{\pi_\mu*\pi_\lambda}\\
\end{array}
$$
\rev{and, for $\mB_\lambda\ni\pi\rv{\underset{\lambda,\mu}\longleftarrow}\eta\in\mB_\mu$, we define $\phi_{\lambda,\mu}(\sfp_\pi\sfp_\eta)\doteq\sfp_{\eta'}\sfp_{\pi'}$ if $(\eta',\pi')$ corresponds to $(\pi,\eta)$ \rv{under} the composition of the above three bijections.}

\rev{Finally let $\lambda_1,\lambda_2,\ldots,\lambda_n$ be dominant weights stabilized by the parabolic subgroup $P$. As seen before for a pair of dominant weights, we define an order $\leq$ on $W^{\lambda_1}\sqcup W^{\lambda_2}\sqcup\cdots\sqcup W^{\lambda_n}$ by: $W^{\lambda_i}\ni\sigma\leq\tau\in W^{\lambda_j}$ if $i\leq j$ and there exist $\sigma',\,\tau'\in W$ such that $\sigma'W_{\lambda_i}=\sigma W_{\lambda_i}$, $\tau' W_{\lambda_j}=\tau W_{\lambda_j}$ and $\sigma'\leq\tau'$ with respect to the Bruhat order on $W$. Further we refine this order to a total order $\leq_t$ such that $W^{\lambda_i}\ni\sigma\leq_t\tau\in W^{\lambda_j}$ only for $i\leq j$.}

\rv{
Let $\leq_{t,\lex}$ be the lexicographic order on the set of words $\mathcal{W}$ defined as follows: $\tau_1\cdots\tau_u\leq_{t,\lex}\sigma_1\cdots\sigma_v$ if there exists $i$ such that $\tau_j=\sigma_j$ for all $j=1,2,\ldots,i$ and either $i=v$ or $i<u,v$ and $\tau_{i+1} <_t \sigma_{i+1}$.
}

\rv{For an L-S path section $\sfp_\pi$ in $\mA\doteq\mA_{\lambda_1}\sqcup\mA_{\lambda_2}\sqcup\cdots\sqcup\mA_{\lambda_n}$ let $\lambda(\sfp_\pi)=\lambda_i$ if $\lambda_i$ is the shape of $\pi$ and extend the shape to monomials as $\lambda(\sfp_{\pi_1}\sfp_{\pi_2}\cdots \sfp_{\pi_u}) = \lambda(\pi_1)+\lambda(\pi_2)+\cdots+\lambda(\pi_u)$.}

\rv{On formal monomials} in \rev{$\mA$} we define the following \rv{order}:
$$
\sfp_{\eta_1}\sfp_{\eta_2}\cdots \sfp_{\eta_u} \rv{\leq} \sfp_{\epsilon_1}\sfp_{\epsilon_2}\cdots \sfp_{\epsilon_v}
$$
if
\rv{
\begin{itemize}
\item[-] the shapes of the two monomials are equal, and
\item[-] $w(\eta_1)w(\eta_2)\cdots w(\eta_u)\leq_{t,\lex}w(\epsilon_1)w(\epsilon_2)\cdots w(\epsilon_v)$.
\end{itemize}
}
Notice \rv{that this order, for our purpose,} is equivalent to the one used in Section~7 of \cite{LLM} (see Proposition~32 in \cite{chirivi_ls} and Proposition~2.1 in \cite{chirivi_multicone}) \rv{since the relations we are going to see are homogeneous with respect to the shape. Further this order verifies the conditions in Section \ref{sec:multiGradedSMT} for an order on monomials.}

Finally notice that we may define a relation $\longleftarrow$ on \rev{$\mA$} by declaring $\rev{\sfp_\pi\longleftarrow\sfp_\eta}$ if $\pi\in\mB_{\lambda_i}$, $\eta\in\mB_{\lambda_j}$ and $\pi\underset{\lambda_i,\lambda_j}\longleftarrow\eta$.

Now consider the $\mk$--algebra
$$\rev{A(\lambda_1,\lambda_2,\ldots,\lambda_n)}\doteq\bigoplus\Gamma(G/P,\calL_{m_1\lambda_1+m_2\lambda_2+\cdots+m_n\lambda_n})$$
where the sum runs over all $n$--\rv{tuples} of \rv{non-negative} integers $m_1,m_2,\ldots,\rev{m_n}$. This algebra is the coordinate ring of the multicone over the partial flag variety $G/P$ mapped diagonally in $\mP(V_{\lambda_1})\times\cdots\times\mP(V_{\lambda_n})$.


\rv{
Now we need a result about the sections $\sfp_\pi$ in order to state our main theorem about standard monomial theory.

\begin{proposition}[See Proposition~7.3 in \cite{LLM}]\label{littelmann_allagrande}  If $\pi_1,\pi_2\,\ldots,\pi_N$ are L-S paths in $\mB_{\lambda_1}\sqcup\mB_{\lambda_2}\sqcup\cdots\mB_{\lambda_n}$, then $\sfp_{\pi_1}\sfp_{\pi_2}\cdots\sfp_{\pi_N} = \sum a_{\eta_1,\eta_2,\ldots,\eta_N}\sfp_{\eta_1}\sfp_{\eta_2}\cdots\sfp_{\eta_N}$, where $\sfp_{\eta_1}\sfp_{\eta_2}\cdots\sfp_{\eta_N}$ is standard and $a_{\eta_1,\eta_2,\ldots,\eta_N}\neq0$ only if $\sfp_{\pi_1}\sfp_{\pi_2}\cdots\sfp_{\pi_N}\rv{\leq} \sfp_{\eta_1}\sfp_{\eta_2}\cdots\sfp_{\eta_N}$.
\end{proposition}

In \cite{LLM}, this proposition is stated and proved only for $N=2$ and $n=1$ (i.e. for products of two sections of the same shape). However the proof there may be verbatim generalized.

We finally have all we need to \rv{state the main result of standard monomial theory for the multicone.}

\begin{theorem}[Proposition~4.1 in \cite{chirivi_multicone}]\label{teo:SMT_multicone}
The set of generators $\mA$, the swap maps $\phi_{\lambda_i,\lambda_j}$ and the relation $\longleftarrow$ define a standard monomial theory for the algebra $A(\lambda_1,\lambda_2,\ldots,\lambda_n)$ on the multiset $\mA$. With respect to the \rv{order $\leq$}, any non-standard monomial in the $\sfp_\pi$'s has a straightening relation.
\end{theorem}

We stress that for this standard monomial theory we cannot apply Corollary \ref{cor:generatedDegreeTwo} since the notion of weakly standard and standard do not coincide. Indeed, in general, there exist minimally non-standard monomials of degree $3$ (see Example \ref{subsectionExampleMulticoneA} below). However, as proved in Proposition~2 in \cite{kempf}, the ideal of relations is still generated in degree two. In particular, it is generated by the straightening relations for the non-standard monomials of degree $2$.

We point out that in the proof of Proposition~4.1 in \cite{chirivi_multicone} there is a slight inaccuracy. Only the proof that the relations of degree $2$ are straightening relations is correct as given there; indeed, in \cite{chirivi_multicone}, Proposition~7.3 in \cite{LLM} is used while one needs its generalization in Proposition \ref{littelmann_allagrande}.
}

\rv{
\subsection{Example of multicones for type $\sfA$}\label{subsectionExampleMulticoneA}

We see an example of a standard monomial theory on a multiset. We develop first some combinatorics about rows and tableaux and then we apply these to the multicones over partial flag varieties for $\sfS\sfL_{\ell+1}$. More details and all proofs about the combinatorics may be found in \cite{chirivi_jeuDeTaquin}, while one may see \cite{chirivi_multicone} about the application to the multicones. For the particular type of multicones we are going to discuss, our standard monomial theory is completely explicit.

We fix a positive integer $\ell$ and denote by $\sfT(k)$ the set of increasing sequences $1\leq i_1<i_2<\cdots<i_k\leq\ell+1$ of integers; we call such a sequence a \emph{row} while $k$ is its \emph{shape}.

We define a (partial) order $\longleftarrow$ on the set of rows in the following way: if $R=i_1i_2\cdots i_k\in\sfT(k)$ and $S=j_1j_2\cdots j_h\in\sfT(h)$ then $R\longleftarrow S$ if either (i) $k\geq h$ and $i_1\leq j_1$, $i_2\leq j_2$, $\ldots$, $i_h\leq j_h$ or (ii) $k\leq h$ and $i_1\leq j_{h-k+1}$, $i_2\leq j_{h-k+2}$, $\ldots$, $i_k\leq j_h$. This order may simply be described as follows. Align the two rows $R,\,S$ to the left if $R$ has shape greater than or equal to that of $S$ or to the right otherwise, then compare the numbers in the columns: if these numbers are non-decreasing then $R\longleftarrow S$. For example $135\longleftarrow 14$ while $45\centernot\longleftarrow135$.

Suppose we are given two rows $R,\,S$, with $R\longleftarrow S$, of shapes $k\leq h$, respectively. One can prove that the set of subrows $S'$ of $S$ of shape $k$ such that $R\longleftarrow S'$ has a minimum $S^0$ for the order $\longleftarrow$. In the same way, there exists a maximum $R^0$ for $\longleftarrow$ in the set of rows $R'$, of shape $h$, containing $R$ and such that $R'\longleftarrow S$. Further $R^0\longleftarrow S^0$ and $R^0\cup S^0=R\cup S$ counting entries with multiplicities.

Analogously, if the two rows $R,\,S$ have shapes $k\geq h$, respectively, we define $R^0$ \rv{as} the $\longleftarrow$--maximal subrow of $R$, of shape $h$, which is $\longleftarrow$--less or equal to $S$, and $S^0$ \rv{as} the minimum of the rows containing $S$, of shape $k$, and $\longleftarrow$--greater or equal to $R$. Also in this case we have $R^0\longleftarrow S^0$ and $R^0\cup S^0=R\cup S$ with multiplicities.

So we have defined a \emph{swap map} $\phi_{k,h}$ for pairs of comparable rows by defining: $\phi_{k,h}(R,S)=(R^0,S^0)$. We have $\phi_{k,k}=\Id$ and $\phi_{h,k}\phi_{k,h}=\Id$. For example $\phi_{2,4}:(25,1346)\longmapsto(1245,36)$.

A sequence $T=R_1,R_2,\ldots, R_N$ of rows is called a \emph{(skew) tableau}; we think to its rows as aligned by the above recipe, each one with respect to the following one. The \emph{shape} of $T$ is the sequence $(k_1,k_2,\ldots,k_N)$ of the shapes of its rows and we denote by $\sfT(k_1,k_2,\ldots,k_N)$ the set of all tableaux of a given shape. The tableau $T$ is \emph{weakly standard} if $R_1\longleftarrow R_2\longleftarrow\cdots\longleftarrow R_N$, i.e. if the numbers in its columns are non-decreasing. For example $24,134,2$ is a weakly standard tableau.

Suppose in the above tableau $T$ we have $R_i\longleftarrow R_{i+1}$, we may then define a new tableau $\tau_i(T)\doteq R_1,\ldots,R_{i-1},R_i^0,R_{i+1}^0,R_{i+2},\ldots,R_N$ where $(R_i^0,R_{i+1}^0)=\phi_{k_i,k_{i+1}}(R_i,R_{i+1})$; the tableau $\tau_i(T)$ has shape $(k_1,\ldots,k_{i-1},k_{i+1},k_i,k_{i+2},\ldots,k_N)$. In particular $\tau_i(T)$ is defined if $T$ is weakly standard.

Now suppose that $T$ is weakly standard and that also $\tau_i(T)$ is weakly standard, then we may define $\tau_j(\tau_i(T))$ by swapping two other rows. If all tableaux \rv{that} we obtain by applying the $\tau_i$'s to $T$ are weakly standard, then we say that $T$ is a \emph{standard tableau}. For example $24,134,3$ is a standard tableau while $24,134,2$ is \emph{not} a standard tableau, indeed if we swap its first two rows we have $124,34,2$ which is not weakly standard.

We denote by $\sfS\sfT(k_1,k_2,\ldots,k_N)$ the set of standard tableaux of shape $(k_1,k_2,\ldots,k_N)$. Further let
$$
\sfS\sfT\{k_1,k_2,\ldots,k_N\}\doteq\bigcup\sfS\sfT(k_{\tau(1)},k_{\tau(2)},\ldots,k_{\tau(N)})
$$
where $\tau$ runs over all permutations in the symmetric group $\sfS_N$. The maps $\tau_i$'s are defined on $\sfS\sfT\{k_1,k_2,\ldots,k_N\}$ and they give an action of $\sfS_N$ on this set.


Let us fix for the sequel a shape $\sfk = (\bar k_1,\bar k_2,\ldots,\bar k_n)$, called the \emph{reference shape}. We say that a shape $(k_1,k_2,\ldots,k_N)$ is \emph{adapted} to the reference shape $\sfk$ if: (i) for all $1\leq i\leq N$ there exists $j_i$ such that $k_i=\bar k_{j_i}$ and (ii) $j_1\leq j_2\leq\cdots\leq j_N$. In the same way, we say that a tableau has \emph{adapted shape} if its shape is adapted to $\sfk$.

Given two tableaux $T=R_1,R_2,\ldots,R_N$ and $T'=R_1',R_2',\ldots,R'_N$ with adapted shapes we define $T\rv{\leq} T'$ if:
\begin{itemize}
\item[-] $T$ and $T'$ have the same shape, and
\item[-] either $T=T'$ or there exists an index $j$ such that
\begin{itemize}
\item[i)] $R_1=R_1',\,R_2=R_2',\ldots,R_j=R_j',\,R_{j+1}\neq R_{j+1}'$ and
\item[ii)] $R_{j+1}\longleftarrow R_{j+1}'$.
\end{itemize}
\end{itemize}

Now we see how the combinatorial data seen above is linked to the standard monomial theory. Let $G=\sfS\sfL_{\ell+1}(\mC)$, $B$ its Borel subgroup of upper triangular matrices, let $\omega_1,\omega_2,\ldots,\omega_{\ell}$ be the fundamental weights numbered as in \cite{bourbaki} and let $P\supseteq B$ be a parabolic subgroup stabilizing the fundamental weights $\omega_{\bar k_i}$ for $i=1,2,\ldots,n$, where $\sfk=(\bar k_1,\bar k_2,\ldots,\bar k_n)$ is the above fixed reference shape.

We want to describe a standard monomial theory for the $\mC$--algebra
$$
A\doteq\bigoplus\Gamma(G/P,\calL_{m_1\omega_{\bar k_1}+m_2\omega_{\bar k_2}+\cdots+m_{n}\omega_{\bar k_n}})
$$
where the sum runs over all $n$--tuples of non-negative integers $m_1,m_2,\ldots,m_n$. This is the same algebra previously studied in this section once we choose $\lambda_1=\omega_{\bar k_1},\,\lambda_2=\omega_{\bar k_2},\ldots,\lambda_n=\omega_{\bar k_n}$.

The set of rows $\sfT(k)$ is in bijection with the L-S paths of shape $\omega_k$, so we have a map $\sfT(k)\ni R\longmapsto\sfp_R\in\Gamma(G/P,\calL_{\omega_k})$. These sections $\sfp_R$'s are nothing else but the classical Pl\"ucker coordinates for the Grassmannian of $k$--dimensional subspaces in $\mC^{\ell+1}$ pulled back to $G/P$.

The order $\longleftarrow$ for rows is the same order defined in the general part in this section by lifting the Bruhat order. The swap maps $\phi_{h,k}$ correspond to the general swap maps (for L-S path sections) and may be defined on sections by: $\phi_{\omega_k,\omega_k}(\sfp_R,\sfp_S)=(\sfp_{R^0},\sfp_{S^0})$ if $\sfT(k)\ni R\longleftarrow S\in\sfT(h)$ and $\phi_{k,h}(R,S)=(R^0,S^0)$.

Notice that if the reference shape is decreasing then a tableau \revtwo{with adapted shape} is standard if and only if it is weakly standard. This is clear since the set of weakly standard tableaux of decreasing shape are a particular instance of a path model. So, in order to check that a weakly standard tableau $T$ is standard one may use the swap maps and make its shape decreasing, obtaining a new tableau $T'$, and then check that $T'$ is (weakly) standard, i.e. check whether the entries of $T'$ are not decreasing in the columns. Further, if $T'$ is standard, then it is uniquely determined by $T$ since the swap maps give an action of the symmetric group. Of course all of this is true also for increasing reference shapes.

If $\sfp_{R_1}\sfp_{R_2}\cdots\sfp_{R_N}$ is a (commutative) monomial in $A$, we may always assume that the shape of the tableau $R_1,R_2,\ldots,R_N$ is adapted to $\sfk$; so we compare monomials in the $\sfp_R$'s via the order \rv{$\leq$} on the corresponding adapted tableaux. This order corresponds to the order defined via the lexicographic order $\leq_{t,\lex}$ on words associated to L-S paths in our situation.

So the standard monomial theory for $A$ may be seen in terms of rows and tableaux. But also its straightening relations are quite explicit. In the sequel we give a set of generators for the ideal $\calR$ of relations among the generators $\sfp_R$'s, i.e. $\calR$ is the kernel of the natural map from the polynomial algebra
$$
\sfS[\sfp_R\,|\,R\textrm{ a row of shape in }\sfk]
$$
to $A$.

We need a slight generalization of these generators: let $R=i_1i_2\cdots i_k$ be any sequence of integers in $\{1,2,\ldots,\ell+1\}$, we define $[R]$ either as $0$ if the entries of $R$ are not distinct, or as $(-1)^\sigma\sfp_{R'}$ if $\sigma$ is the unique permutation of $1,2,\ldots,k$ such that $R'=i_{\sigma(1)},i_{\sigma(2)},\ldots,i_{\sigma(k)}$ is a row.

Now suppose that $(k,h)$, with $k>h$, is a shape adapted to the reference shape and suppose we have two rows
$$
\begin{array}{cccccccccccc}
R & = & i_1 & i_2 & \ldots & i_{t-1} & u_{t+1} & u_{t+2} & \cdots & \cdots & \cdots & u_{k+1}\\
S & = & u_1 & u_2 & \ldots & u_{t-1} & u_t & j_1 & j_2 & \cdots & j_{h-t}\\
\end{array}
$$
with shapes $k,\,h$, respectively, such that $i_1<u_1,\,i_2<u_2,\ldots,i_{t-1}<u_{t-1}$ but $u_{t+1}>u_t$ so that $T\doteq R,S$ is not a standard tableau; we say that $t$ is the \emph{index of violation} of standardness in $T$. Then the polynomial
$$
\sum (-1)^\sigma[i_1,\ldots,i_{t-1},u_{\sigma(t+1)},\ldots,u_{\sigma(k+1)}][u_{\sigma(1)},\ldots,u_{\sigma(t)},j_1,\ldots,j_{h-k}]
$$
where the sum runs over a set of representatives for the quotient $\sfS_{k+1}/\sfS_t\times\sfS_{k+1-t}$, is in $\calR$. Such a relation is called a \emph{shuffling relation}. The case of decreasing adapted shape $k<h$ results in similar shuffling relations.

Notice that a shuffling relation may not be a straightening relation; indeed other non-standard tableaux besides $T$ (corresponding to $\sigma\,\,=$ the identity permutation) may appear. But any other non-standard tableau appearing in this relation has index of violation greater than $t$. So we may use a finite number of shuffling relations and reach eventually a straightening relation for $T$.

Finally, since standard and weakly standard coincide for tableaux with two rows, and since, by \cite{kempf}, we know that $\calR$ is generated in degree $2$, we conclude that the shuffling relations generate $\calR$.

Let us see an example with $\ell = 3$ and reference shape $\sfk=(2,3,1)$. The tableau $T\doteq24,134,2$ is weakly standard but not standard (as already seen above). So $[24][134][2]=\sfp_{24}\sfp_{134}\sfp_{2}$ is a linear combination of sections associated to standard tableaux.

We have the following shuffling relations (in particular they are also straightening relations):
$$
\begin{array}{ccccccc}
[234][14] & - & [134][24] & + & [124][34] & = & 0\\
\end{array}
$$
$$
\begin{array}{ccccccc}
[34][2] & - & [24][3] & + & [23][4] & = & 0\\
\end{array}
$$
If we multiply the first one by $[2]$, use the second and move $T$ to the left hand side, we have
$$
\begin{array}{ccccccc}
T & = & [14][234][2] & + & [24][124][3] & - &[23][124][4]\\
\end{array}
$$
As one can easily check, the three tableaux in the right hand side are all standard; hence we have obtained the straightening relation for $T$.

Now let $\calR_0$ be the ideal of $\sfS[\sfp_R]$ generated by $\sfp_{R_1}\sfp_{R_2}\cdots\sfp_{R_N}$ for all non-standard tableaux $R_1,R_2,\ldots,R_N$. The quotient $A_0=\sfS[\sfp_R] / \calR_0$ is called the \emph{discrete algebra} for the multicone with reference shape $\sfk$. Notice that it is possible to define a certain valuation and, using Theorem \ref{thm:degeneration}, degenerate $A$ to $A_0$. In particular, in our example above with $\ell=3$ and $\sfk=(2,3,1)$, the ideal $\calR_0$ is no more generated in degree $2$ since the tableau $T$, for example, is weakly standard but not standard. The same is true for any non-decreasing or non-increasing reference shape.

\subsection{A multicone for type $\sfA_1\times\sfA_1$}\label{subsectionAnotherExample}

Now we see a very simple example of multicone which will be used in the next Section. 
Let $G\doteq\sfS\sfL_2(\mC)\times\sfS\sfL_2(\mC)$, a group of type $\sfA_1\times\sfA_1$, 
let $\omega,\,\omega'$ be the two fundamental weights and take $\lambda_1\doteq\omega,\,\lambda_2=\omega+\omega'$ and $\lambda_3=\omega'$ 
(we are using the symbols $\lambda_1,\,\lambda_2,\,\lambda_3$ with the same meaning as in the main part of this Section). The multicone for this example is $\sfS\sfL_2(\mC)/B\times\sfS\sfL_2(\mC)/B$, which is clearly isomorphic to $\mP^1\times\mP^1$. We will describe the combinatorics and the straightening relations using tableaux with rows made of two boxes with some boxes filled with the integers $1$ and $2$.

In particular here is the correspondence between rows and L-S paths for $\mB_{\lambda_i}$:
$$
\begin{array}{ccccc}
\Young{1\hfill} & \longmapsto & \pi_{\omega} & \in & \mB_{\lambda_1}\\
\Young{2\hfill} & \longmapsto & \pi_{-\omega} & \in & \mB_{\lambda_1}\\
\Young{11} & \longmapsto & \pi_{\omega+\omega'} & \in & \mB_{\lambda_2}\\
\Young{21} & \longmapsto & \pi_{-\omega+\omega'} & \in & \mB_{\lambda_2}\\
\Young{12} & \longmapsto & \pi_{\omega-\omega'} & \in & \mB_{\lambda_2}\\
\Young{22} & \longmapsto & \pi_{-\omega-\omega'} & \in & \mB_{\lambda_2}\\
\Young{\hfill1} & \longmapsto & \pi_{\omega'} & \in & \mB_{\lambda_3}\\
\Young{\hfill2} & \longmapsto & \pi_{-\omega'} & \in & \mB_{\lambda_3}\\
\end{array}
$$
In the sequel we write $\Young{i}$ where $i$ may be $1$, $2$ or nothing and we denote pairs of rows and tableaux as stacked rows of two boxes. Notice that by $\Young{ij}$ we mean one of the rows in the above list, so that the row $\Young{\hfill \hfill}$ is not allowed in our context.

We define the following relation on boxes:
$$
\Young{i}\leq\Young{j}
$$
for all pairs $(i,j)$ with $i,\,j$ equal to $1$,$2$ or nothing, but the pair $i=2$ and $j=1$.
The relation $\longleftarrow$, defined by lifting the Bruhat order, corresponds to the following relation via the above bijection from rows to paths:
$$
\Young{ij}\longleftarrow\Young{hk}\quad\quad \textrm{if and only if}\quad\quad \Young{i} \leq \Young{h} \,\,\textrm{ and }\,\, \Young{j} \leq \Young{k}
$$
Hence a tableau 
$$
T={\lower30pt\hbox\Young{\ione\jone,\itwo\jtwo,\cdot\cdot,\cdot\cdot,\cdot\cdot,\ienne\jenne}}
$$
is weakly standard if and only if: in each column, consecutive integers do not decrease.

The crystal graph isomorphisms are very easy to compute and, denoting by $i,j,k$ integers in $\{1,\,2\}$ (so that $\Young{i}$, $\Young{j}$ and $\Young{k}$ are not empty boxes), the resulting swap maps are:
$$
\begin{array}{ccc}
\YYoung{ij,k\hfill} & \longleftrightarrow & \YYoung{i\hfill,kj}\\
\\
\YYoung{\hfill i,jk} & \longleftrightarrow & \YYoung{ji,\hfill k}\\
\\
\YYoung{\hfill i,j\hfill} & \longleftrightarrow & \YYoung{j\hfill,\hfill i}\\
\end{array}
$$
The swap maps may be summarized in words as: vertically exchange the empty boxes with the filled ones. Hence a tableau $T$ as above is standard if in each column the integer entries, read out by skipping the empty boxes, are non-decreasing. In particular, if in a tableau the empty boxes are only $j_1,j_2,\ldots,j_h$ and $i_k,i_{k+1},\ldots,i_N$ for certain $h,\,k$, then it is standard if and only if it is weakly standard.

The coordinate ring of the multicone is
$$
A\doteq\bigoplus_{n_1,n_2,n_3}\Gamma(\sfS\sfL_2(\mC)/B\times\sfS\sfL_2(\mC)/B,\calL_{(n_1+n_2)\omega+(n_2+n_3)\omega'})
$$
and we have a surjective map from the polynomial algebra $\sfS(\Young{ij})$ with indeterminates indexed by the rows \Young{ij}, to $A$ whose kernel $\calR$ is generated by the polynomials
$$
\begin{array}{ccc}
\YYoung{2\hfill,1i} - \YYoung{1\hfill,2i}\\
\\
\YYoung{21,12} - \YYoung{11,22}\\
\\
\YYoung{i2,\hfill 1} - \YYoung{i1,\hfill 2}\\
\end{array}
$$
where $i\in\{1,\,2\}$.
}

\section{Standard monomial theory for the Cox ring of a wonderful variety}\label{sec:SMTCox}

Let $X$ be a wonderful $G$--variety with (unique) closed $G$--orbit $Y$, and \rv{let $P \supseteq B$ be the parabolic subgroup} such that $Y \simeq G/P$.
By \cite{luna_spherical}, $X$ is \textit{spherical}, i.e.\ it possesses an open $B$--orbit, say $B\cdot x_0 \subset X$.
Since $B\cdot x_0$ is affine, $G\cdot x_0 \setminus B\cdot x_0$ is a union of finitely many $B$--stable \rv{divisors} and we denote by $\grD$ the set of their closures in $X$:
\[ \grD \doteq \{ D \subset X \st D \text{ is a $B$--stable prime divisor, } D \cap G\cdot x_0 \neq \vuoto \}.\]
The elements of $\grD$ are called the \textit{colors} of $X$.

Denote by $B^-$ the opposite Borel subgroup of $B$ and let $y_0 \in Y$ be the unique $B^-$--fixed point of $X$. The normal space \rv{$\mathrm T_{y_0} X/\mathrm T_{y_0} Y$} of $Y$ in $X$ at $y_0$ is a multiplicity-free $T$--module. The elements of the set
\[ \grS \doteq \left\{ \textrm{$T$--weights of } \mathrm T_{y_0} X/\mathrm T_{y_0} Y \right\}\]
are called the \textit{spherical roots} of $X$. \rv{If $\sigma\in\Sigma$, there exists a unique $G$--stable \rv{divisor} $X_\sigma$ of $X$ such that the weight of $T$ on $\mathrm T_{y_0} X/\mathrm T_{y_0} X_\sigma$ is $\sigma$. This gives a natural correspondence between the set $\Sigma$ and the irreducible boundary divisors of $X$.} 

Recall that every line bundle on $X$ \rev{and} on $Y$ has a unique $G$--linearization.
As a group, $\Pic(X)$ is freely generated by the equivalence classes of line bundles $\calL_D\doteq\calO(D)$, for $D\in\grD$ (see \cite[Proposition~2.2]{brion_picard}). For all $E\in\mZ\grD$, the associated line bundle $\calL_E\doteq\calO(E)$ is globally generated, respectively ample, if and only if $E$ is a non-negative, respectively positive, combination of colors. \rv{Notice that $\mZ\Sigma$ is a sublattice of $\mZ\Delta$.}

The restriction of line bundles to the closed orbit induces a map $\grl : \Pic(X) \lra \grL$; given $E \in \mZ \grD$ we set $\grl_E \doteq \grl(\calL_E)$ in such a way that $\Gamma(Y,\rrev{\calL_E\ristretto_Y})\simeq V_{\lambda_E}^*$ and\rev{, moreover,} we set $V_E\doteq \rev{V_{\lambda_E}^*}$ for short. \rev{(Hence $\rrev{\calL_E\ristretto_Y}\simeq\calL_{\lambda_E}$ where this last line bundle is defined in the previous section.)} Moreover, in particular, $\Gamma(X,\calL_E)$ contains a copy of $\rev{V_E}$ and, \rv{since $X$ is spherical} the decomposition of \rv{the $G$--module} $\Gamma(X,\calL_E)$ is multiplicity-free.

%
%

%
%


If $\grg \doteq \sum a_\grs \grs\in \mN\grS$, we denote by $s^\grg \in \grG(X,\calL_{X^\grg})$ a section whose divisor is equal to $\rv{X_\gamma\doteq\sum a_\grs X_\grs}$; notice that this section is $G$--invariant. If $E,F \in \mZ \grD$ are such that $F-E \in \mN\grS$, then we write $E \leq_\grS F$. If $E\in \mN \Delta\rev{, F\in\mZ\Delta}$ and $E\leq_\grS F$ the multiplication by $s^{F-E}$ induces a $G$--equivariant map from the sections of $\calL_E$ to the sections of
$\calL_F$, in particular we have $s^{F-E}\rev{V_E}\subseteq \Gamma(X,\calL_F)$. Moreover

\begin{proposition}[{\cite[Proposition~2.4]{brion_picard}}]	\label{prop: decomposizione sezioni}
Let $F \in \mZ\grD$, then 
\[
	\grG(X,\calL_F)\quad = \bigoplus_{E \in \mN \grD \st E \leq_\grS F} s^{F-E} \rev{V_E}.
\]
\end{proposition}

\rv{Since $\Pic(X)\simeq\mZ\grD$ is a free lattice, the space}
$$
C(X) \doteq\bigoplus _{D\in \mZ\grD} \Gamma(X,\calL_D)
$$
is a ring; in analogy with the toric case $C(X)$ is called the \emph{Cox ring} of $X$. The ring $C(X)$ was studied in \cite{CM_ring} and \cite{CLM} in the case of a wonderful symmetric variety (where it is called respectively the \textit{ring of sections} of $X$ and the \textit{coordinate ring} of $X$), and in \cite{brion_cox} in the case of a wonderful variety (where it is called the \textit{total coordinate ring} of $X$).

Since $X$ is irreducible, \rv{by Proposition \ref{prop: decomposizione sezioni},} $C(X)$ is generated as a $\mk$--algebra by the sections $s^\grs$\rv{, for $\sigma\in\Sigma$}, and by the modules $\rev{V_D} \subseteq \Gamma(X,\calL_D)$ for $D\in \Delta$. It follows that $C(X)$ is a quotient of the \rrev{symmetric algebra}
$$
	\rev{\sfS(X)} \doteq \mk[s_1, \ldots, s_r] \otimes \mathsf{S}\Big(\bigoplus_{D \in \grD} \rev{V_D} \Big)
$$
where we fix an ordering $\grS = \{\grs_1, \ldots, \grs_r\}$ and we set $s_i \doteq s^{\grs_i}$ for short. Further, notice that the quotient of $C(X)$ by the ideal generated by the sections $s_1, \ldots, s_r$ is isomorphic to the coordinate ring of a multicone over the flag variety $Y \simeq G/P$, that is
$$
	\rev{C(Y)\doteq A(\lambda_{D_1},\lambda_{D_2},\ldots,\lambda_{D_q})} = \bigoplus_{D \in \mN \grD} \Gamma(Y,\calL_D\ristretto_{Y}) \simeq \bigoplus_{D \in \mN \grD} \rev{V_D}.
$$
where $\grD = \{D_1, \ldots, D_q\}$ is any fixed ordering of $\grD$. Therefore we have surjective maps
$$
	\rev{\sfS(X)} \lra C(X) \lra \rev{C(Y)},
$$
The rings $\rev{\sfS(X)}$, $C(X)$ and $\rev{C(Y)}$ all have natural $\mZ\grD$--gradings, and the previous maps are morphisms of $\mZ\grD$--graded $G$--algebras.

By Theorem \ref{teo:SMT_multicone} we have a standard monomial theory with straightening relations for $\rev{C(Y)}$. 
Our aim is \rev{to extend it to} a standard monomial theory for the Cox ring $C(X)$, and deduce a degeneration result for such a ring. 
\revi{A description of the ideal $\calI_X$ defining $C(X)$ as a quotient of $\sfS(X)$ is given in terms of straightening relations of our standard monomial theory.}

Given $D \in \grD$ we denote by $\mB_D \doteq \mB_{\lambda_D}$ the set of L-S paths of shape $\grl_D$ for short. Given $\pi \in \mB_D$, let $\sfx_\pi \in \rev{V_D} \subset \grG(X, \calL_D)$ be the unique section such that $\sfx_\pi\ristretto_Y = \sfp_\pi$ and let
$$
	\mA_D \doteq \{\sfx_\pi \st \pi \in \mB_D\}\subset \rev{V_D}.
$$
Then $\mA_D$ is a basis of $\rev{V_D}\subseteq\Gamma(X,\calL_D)$. Further, define
$$
	\mA_\grS \doteq \{s_i \st i=1,\dots,r\}\text{, }\quad \mA_\grD \doteq \rev{\bigsqcup_{D \in \grD}} \mA_D\quad\textrm{ and }\mA_X \doteq \mA_\grS \rev{\sqcup} \mA_\grD.
$$
In particular, $\rev{\sfS(X)}$ is the \rrev{symmetric} algebra in the \rrev{indeterminates} $\mA_X$. \revi{The multiset of our standard monomial theory is $\mA_X=\mA_\Sigma\sqcup \bigsqcup_{D \in \Delta} \mA_D$ so that if $\sfx_\pi \in \mA_\grD$, its \textit{shape} is the unique $D \in \grD$ such that $\sfx_\pi \in \mA_D$ and the shape of $s_i$  is $\Sigma$. \revtwo{Recall that in the definition of a standard monomial theory, we have to fix an order on the possible shapes.} Here we have already fixed a total order on $\Delta$ and we extend it by declaring $\Sigma<D$ for all $D\in \Delta$.}

Let $\mB_\grD \doteq \rev{\bigsqcup_{D \in \grD}}\mB_D$ and set
$$
	\mA_Y \doteq \{\sfp_\pi : \pi \in \mB_\grD\},
$$
which is naturally identified with the subset $\mA_\grD \subset \mA_X$ via the bijection $\sfx_\pi \longmapsto \sfp_\pi  = \sfx_\pi\ristretto_Y$.

By Theorem \ref{teo:SMT_multicone} we have a standard monomial theory for $\rev{C(Y)}$. We denote by $\sfM(Y) \subset \sfS(\mA_Y) \simeq \sfS(\mA_\grD)$ the set of monomials in the coordinates $\mA_Y$, and by $\SM(Y) \subset \sfM(Y)$ the subset of standard monomials. In particular, using the bijections $\mA_D \simeq \mA_{\grl_D}$, for all $D, D' \in \grD$ we have swap maps $\phi_{D,D'}$, a relation $\longleftarrow$ \rev{on $\mA_\Delta$} and an \rv{order $\leq$} on $\sfM(Y)$ as defined in the previous section.

First we extend the relation $\longleftarrow$ \rev{to $\mA_X$} by declaring $s_1\longleftarrow s_2\longleftarrow\cdots\longleftarrow s_r$ and $s_i \longleftarrow \sfx_\pi$, $\sfx_\pi \longleftarrow s_i$ for all $i=1,2,\ldots,r$ and all $\pi\in\mB_\Delta$. Next we extend the swap maps by
\begin{align*}
	\phi_{\grS, D}(s_i, \sfx_\pi) \doteq (\sfx_\pi, s_i) \qquad \qquad
	\phi_{D,\grS}(\sfx_\pi, s_i) \doteq (s_i, \sfx_\pi)
\end{align*}
for all $i=1,2,\ldots,r$ and all $\pi \in \mB_\Delta$.

\rv{
Let $\sfm \doteq s^\grg \sfx_{\pi_1} \cdots \sfx_{\pi_N}$ be a monomial\rv{; then} $\nu(\sfm) \doteq \grg$ is called the \textit{vanishing} of $\sfm$
\rv{. Further if} $D_i \in \grD$ is the shape of $\sfx_{\pi_i} \in \mA_\grD$, we define the \revi{\textit{Picard weight}} of $\sfm$ as
$$
\grg + \sum_{i=1}^N D_i,
$$
namely the degree of $\sfm$ with respect to the $\mZ\grD$--grading.
}

Finally, we extend the order $\leq$ to monomials in $\mA_X$: if $\sfm_1,\sfm_2$ are two monomials in $\sfS(\mA_{\grD}) \simeq \sfS(\mA_Y)$ and if $\gamma_1,\,\gamma_2\in\mN\Delta$, then we set $s^{\gamma_1}\sfm_1\rv{\leq} s^{\gamma_2}\sfm_2$ if:\rv{
\begin{itemize}
\item[-] the \revi{Picard weights} of $s^{\gamma_1}\sfm_1$ and of $s^{\gamma_2}\sfm_2$ are equal, and
\item[-] either $\gamma_1<_\Sigma\gamma_2$ or $\grg_1 = \grg_2$ and $\sfm_1\ristretto_Y \rv{\leq} \sfm_2\ristretto_Y$ (with respect to the order $\leq$ defined in Section \ref{sec:SMTMulticoneFlag}).
\end{itemize}
}
We denote by $\sfM(X) \subset \sfS(X)$ the set of the monomials in the indeterminates $\mA_X$, endowed with the \rv{order $\leq$}.


The inclusion $\mA_Y \longhookrightarrow \mA_X$ defines a \rrev{shape-preserving} bijection between $\sfM(Y)$ and the subset of $\sfM(X)$ of the monomials $\sfm$ such that $\nu(\sfm) = 0$. \rev{Given} $\sfn \in \sfM(Y)$ we denote by $\wt \sfn \in \sfM(X)$ the corresponding monomial, so we have $\wt \sfn \ristretto_Y = \sfn$ \rev{and, in particular, $\wt\sfp_\pi=\rrev{\sfx_\pi}$}. Conversely, given a monomial $\sfm \in \sfM(X)$ we may define a monomial $\ol\sfm \in \sfM(X)$ with $\nu(\ol\sfm) = 0$ by setting $\ol\sfm \doteq s^{-\nu(\sfm)} \sfm$. Notice that $\sfm \in \sfM(X)$ is a standard monomial if and only if $\ol\sfm \in \sfM(X)$ is a standard monomial, if and only if $\ol\sfm\ristretto_Y \in \sfM(Y)$ is a standard monomial. We denote by $\SM(X) \subset \sfM(X)$ the set of standard monomials in $\mA_X$, and if $E \in \mZ \grD$ we denote by $\SM_E(X)\subset\rev{\sfM_E}(X)$ the set of standard monomials and \rv{the set} of all monomials of \revi{Picard weight} $E$, respectively.

Following \cite[Theorem 3]{CM_ring}, we are now able to construct a standard monomial theory for the Cox ring $C(X)$ of a wonderful variety $X$.

\begin{theorem}\label{teo:SMT_cox}
\begin{itemize}
\item[i)] Given $E \in \mZ \grD$, the images of the standard monomials of \revi{Picard weight} $E$ form a basis of $\grG(X,\calL_E)$.
\item[ii)] Given a non-standard monomial $\sfm'$ the equality
\begin{align*}
	\sfm' \quad = \sum_{\sfm \in \SM(X)} a_{\sfm} \sfm
\end{align*}
guaranteed by \emph{i)} is a straightening relation in $C(X)$; that is, we have $\sfm'\rv{\leq} \sfm$ whenever $a_\sfm \neq 0$. Moreover,
\begin{align*}
\ol\sfm' \quad = \sum_{\sfm \in \SM(X), \nu(\sfm)=\nu(\sfm')} a_{\sfm} \ol{\sfm}
\end{align*}
is a straightening relation in $\rev{C(Y)}$.
	\item[iii)] \revi{The defining ideal $\calI_X \subset \rev{\sfS(X)}$ is generated by the straightening relations for the non-standard monomials of degree two.}
\end{itemize}
In particular, the set of generators $\mA_X$ together with the above defined swap maps, relation and \rv{order} define a \rv{standard monomial theory on the multiset $\mA_X$} with straightening relations for the Cox ring $C(X)$.
\end{theorem}
\begin{proof}
We prove the first two statements together. Let $\pi_1, \ldots, \pi_N \in \mB_\grD$ be such that $\sfx_{\pi_1} \cdots \sfx_{\pi_N}$ is not standard. In $A(Y)$, by Theorem \ref{teo:SMT_multicone}, we have a straightening relation
$$
	\sfp_{\pi_1} \cdots \sfp_{\pi_N} = \sum_{\sfn \in \SM(Y)} a_\sfn \sfn,
$$
where $\sfp_{\pi_1} \cdots \sfp_{\pi_N} \rv{\leq} \sfn$ for all $\sfn$ such that $a_\sfn \neq 0$.

Since $X \setminus G \cdot x_0$ is a normal crossing divisor with smooth irreducible components, a section in $C(X)$ vanishes on the closed orbit $Y$ if and only if it is in the ideal generated by the sections $s_1, \ldots, s_r$. By construction the difference $\sfx_{\pi_1} \cdots \sfx_{\pi_N} - \sum_{\sfn} a_\sfn \wt \sfn$ is homogeneous w.r.t. the $\mZ\grD$--grading, and it vanishes on $Y$. Hence we have
\begin{equation*}
	\sfx_{\pi_1} \cdots \sfx_{\pi_N} = \sum_{\sfn \in \SM(Y)} a_\sfn \wt \sfn + \sum_{\sfm \in \sfM_E(X) \st \nu(\sfm) \neq 0} a_\sfm \sfm,
\end{equation*}
where $\sfx_{\pi_1} \cdots \sfx_{\pi_N} \rv{\leq} \wt \sfn$ for all $\sfn \in \SM(Y)$ with $a_\sfn \neq 0$. 

Proceeding inductively on the partial order $\leq_\grS$, the previous equality implies that in $C(X)$ the image of every monomial $\rev{\sfm'} \in \sfM_E(X)$ may be written as the image of a sum of standard monomials $\sfm \in \SM_E(X)$ \rev{with $\sfm'\rv{\leq}\sfm$}.

%
%

Therefore the image \rev{of the standard monomials} of $\SM_E(X)$ in $C(X)$ is a set of generators for $\grG(X,\calL_E)$ as a vector space. On the other hand, by Theorem \ref{teo:SMT_multicone}, the images of the standard monomials $\sfn \in \SM(Y)$ form a basis for $\rev{C(Y)}$; hence for all $F \in \mN \grD$ the images of the standard monomials $\sfn \in \rev{\SM(Y)}$ \revi{of Picard weight $F$} form a basis for the graded component $\rev{C(Y)_F} = \rev{V_F}$. \rev{So using Proposition \ref{prop: decomposizione sezioni}}, we have
$$
\begin{array}{rcl}
\dim \grG(X,\calL_E) & = & \sum_{F \in \mN\grD \st F \leq_\grS E} \dim \rev{V_F}\quad\\
 & = & \sum_{F \in \mN\grD \st F \leq_\grS E} |\SM_F(Y)|\\
 & = & |\SM_E(X)|,
\end{array}
$$
and this finishes the proof of i) and ii).

\revi{Now, in order to prove iii), let $J$ be the ideal of $\rev{\sfS(X)}$ generated by the straightening relations for the non-standard monomials of degree two}. Clearly $J\subseteq \calI_X$ and we want to show that these two ideals are equal.

The quotient $\rev{\sfS(X)}/\langle J,s_1,\ldots,s_r\rangle$ is isomorphic to $\rev{C(Y)}$ since the relations for this last ring \rv{are} generated by the quadratic straightening relations; indeed it is generated in degree $2$ by {\cite[Proposition~2]{kempf}}. So, if $\sfm'$ is a non-standard monomial $\sfm'+\langle s_1,\ldots,s_r\rangle$ is a sum of standard monomials \rev{modulo $J$}. Hence in $\rev{\sfS(X)}/J$ the monomial $\sfm'$ is a sum of standard monomials $\sfm$ with $\nu(\sfm)=0$ plus $s_1y_1+\cdots+s_ry_r$ for some homogeneous elements $y_1,\ldots,y_r$ whose \revi{Picard weights} are $<_\Sigma$ of that of $\sfm'$.

Proceeding again by induction on the \revi{Picard weight} of a non-standard monomial we see that any straightening relation is an element of $J$. So $J=\calI_X$ and the last statement of the \rv{Theorem} is proved.

\end{proof}

\rv{
\begin{remark} The results stated in the above Theorem overlap with Proposition~3.3.1 in \cite{brion_cox}. In that Proposition a description of the quadratic relations of the ring $C(X)$ is given. However in \cite{brion_cox} the standard monomial structure is not considered explicitly. Notice that even if the relations of the ring $C(X)$ are  generated in degree two, the relations of degree two are not enough to construct a standard monomial theory as we have already noticed in the final part of Example \ref{subsectionExampleMulticoneA}.
\end{remark}
}

The standard monomial theory constructed in the previous theorem is compatible with the $G$--orbit closures in $X$ \rv{in the following sense}. Recall that the subsets $I \subseteq \grS$ parametrize the $G$--orbits in $X$; that is, for every $x \in X$ there is a unique $I \subseteq \grS$ such that $\ol{G\cdot x} = \bigcap_{\grs \in \grS \setminus I} X^\grs \doteq X_I$.

The $G$--stable subvariety $X_I$ is again a wonderful variety; its set of spherical roots coincides with $I$. Given $I \subseteq \grS$, we say that a standard monomial $\sfm \in \SM(X)$ is $I$--\textit{standard} if $\nu(\sfm) \in \langle \grS \setminus I\rangle_\mZ$. We denote by $\SM^I(X)$ the set of $I$--standard monomials, and by $\SM^I_E(X)$ the set of the $I$--standard monomials of \revi{Picard weight} $E \in \mZ \grD$. 

\begin{corollary}
Given $E \in \mZ\grD$, the images of the $I$--standard monomials $\SM_E^I(X)$ are a basis for $\grG(X_I,\calL_E\ristretto_{X_I})$.
\end{corollary}

\begin{proof}
Let $J \doteq \grS \setminus I$. Then the restriction of sections $\rho : \grG(X,\calL) \lra \grG(X_I,\calL_E\ristretto_{X_I})$ is a surjective map, and we have
$$
\ker \rho \quad = \bigoplus_{F \in \mN\grD \st F \leq_J E} s^{E-F} \rev{V_F},
$$
where we write $F \leq_J E$ if and only if $E-F \in \mN J$. It follows that the images of the $J$--standard monomials of \revi{Picard weight} $E$ give a basis for $\ker \rho$, whereas the restrictions of the images of the $I$--standard monomials of \revi{Picard weight} $\rrev{E}$ give a basis for $\grG(X_I, \calL_E\ristretto_{X_I})$.
\end{proof}

When $X$ is the wonderful compactification of a semisimple adjoint group regarded as a homogeneous $G \times G$--variety, the above constructed standard monomial theory is even compatible with the $B \times B$--orbit closures (see \cite{appel}).

\revi{
\subsection{An example of the Cox ring for a wonderful variety}
We illustrate our theory in a simple example. We will make use of the results and conventions introduced in Section \ref{subsectionAnotherExample}.
Let $V=\mC^2$ and define
$$
X=\{
([\grf],[A],[v])\in \mP(V^*)\times \mP(\End(V))\times \mP(V) \st \grf(Av)=0\}.
$$
$X$ is a wonderful variety for the group $G=\sfS\sfL(V)\times \sfS\sfL(V)$ acting by $(g,h)(\grf,A,v)= (g\cdot\grf, gAh^{-1}, h \cdot v)$. We
choose the maximal torus  given by the diagonal matrices and the Borel subgroup given by the upper triangular matrices. We denote by $\gra$ the simple root of the first factor of $G$ and by $\gra'$ the simple root of the second factor.

The Picard group is generated by the pull-backs of the three line bundles $\calO_{\mP(V^*)}(1)$,
$\calO_{\mP(\End(V))}(1)$ and  $\calO_{\mP(V)}(1)$ and we denote by $D_1,D_2,D_3$ the associated colors, respectively. We have two spherical roots
$$
\grs_1=\gra=D_1+D_2-D_3 \mand \grs_2=\gra'=D_2+D_3-D_1.
$$
The closed orbit $Y$ in this case is $\mP^1\times \mP^1$ and the ring $C(Y)$ is the one studied in Section \ref{subsectionAnotherExample}.

The ring $C(X)$ is generated by the eight generators $\sfx_\pi$ that we denote with rows of length two as in Section \ref{subsectionAnotherExample}, 
together with $s_1$ and $s_2$. The standard monomials are the monomials of the form $s_1^as_2^b \tilde\sfm$, where $\tilde\sfm$ is a standard monomial for the ring $C(Y)$. In particular we have the same minimally non-standard monomials as in the ring $C(Y)$, and the straightening relations are given by
\begin{align*}
\YYoung{2\hfill,1i}&= \YYoung{1\hfill,2i} + s_1 \Young{\hfill i}\\
\YYoung{21,12} &= \YYoung{11,22} +s_1s_2\\
\YYoung{i2,\hfill 1} &= \YYoung{i1,\hfill 2}+s_2 \Young{i\hfill}
\end{align*}
where $i\in\{1,\,2\}$.
}

\section{Degeneration and \rev{rational singularities}}

%
%

Any straightening relation involves monomials with higher power of the sections $s_1, \ldots, s_r$\rev{.} This allows us to degenerate $\Spec C(X)$ to the product of the affine space $\mk^r$ with a multicone over the flag variety $G/P \simeq Y$. Let us see the details for such a degeneration.

\rev{
\begin{corollary}\label{cor:degeneration_cox} There exists a flat $G\times\mk^*$--equivariant degeneration $\calC$ of $C(X)$ to the ring $\mk[s_1,\dots,s_r]\otimes C(Y)$; further all generic fibers of $\calC$ are isomorphic to $C(X)$.
\end{corollary}
\begin{proof} We define a map $\delta:\mA_X\longrightarrow\mN$ by $\delta(\sfx_\pi)=0$ for all $\pi\in\mB_\Delta$ and $\delta(s_i)=1$ for all $i=1,2,\ldots,r$. This map is a valuation for the standard monomial theory of $C(X)$ by Theorem \ref{teo:SMT_cox}. Hence we may apply Theorem \ref{thm:degeneration}. The special fiber is isomorphic to the ring in the statement of the theorem again by Theorem \ref{teo:SMT_cox}.

Moreover, for this valuation map the Rees algebra is
$$
\calC=\cdots\oplus C(X)t^2\oplus C(X)t\oplus C(X)\oplus Kt^{-1}\oplus K^2t^{-1}\oplus\cdots
$$
with $K$ the ideal of $C(X)$ generated by the sections $s_1,s_2,\ldots,s_r$. So, being $K$ generated by $G$--invariants, the action of $G$ on $C(X)$ induces an action on $\calC$ by letting $G$ act trivially on $t$. In particular $G$ acts on each fiber and it is clear that this $G$--action commutes with the isomorphisms $\calC_a\longrightarrow\calC_{\lambda^{-1}a}$ for any $a\in\mk$ and $\lambda\in\mk^*$. So the deformation is also $G$--equivariant.
\end{proof}
}

We now apply \rev{this} degeneration result to the study of the singularities of the algebra $C(X)$. A variety $X$ is said to have \emph{rational singularities} if there exists a resolution of singularities $\pi:Y\lra X$ of $X$ such that $R^i\pi_*\calO_Y=0$ for $i>0$ and $\pi_*\calO_Y=\calO_X$.  If such a property holds for a resolution then it holds for all resolutions. Finally a ring $A$ is said to have rational singularities if $\Spec A$ has rational singularities.

We have the following properties:
\begin{itemize}
\item[(a)] a multicone over a flag variety has rational singularities (see \cite{kempf}, Theorem~2)
\item[(b)] if $X$ is an affine $G$--variety with rational singularities and $G$ is a reductive group then $X/\!\!/G$ has rational singularities (see \cite{boutot})
\item[(c)] if $(\calX,X)\lra (S,s_0)$ is a flat deformation of a variety with rational singularities $X$ then there \rrev{exists} a neighbourhood $U$ of $s_0$ such that for $s\in U$ the fiber over $s$ has also rational singularities (see \cite{elkik}, \rev{Th\'eor\`eme}~4).
\end{itemize}

Given $D \in \mZ \grD$ consider the subalgebra of $C(X)$ defined as follows
$$
	C_D(X) \doteq \bigoplus_{n\geq 0} \Gamma(X,\calL_{nD}).
$$
This is the projective coordinate ring of a spherical variety, namely \rrev{the image} of $X$ in the projective space $\mP(\grG(X,\calL_D)^*)$. \rrev{It is known that these rings have} rational singularities (see \cite{popov}, or also \cite[Remark 2.5]{alexeev-brion} for another proof which is closer to the constructions of this paper).

\begin{proposition}\label{prp:singraz}
Let $X$ be a wonderful variety and let $D \in \mZ \grD$. Then $C(X)$ and $C_D(X)$ have rational singularities.
\end{proposition}

\begin{proof}
By Corollary \ref{cor:degeneration_cox} we have that $C(X)$ is a deformation of a multicone over a flag variety, which has rational singularities by (a), hence $C(X)$ has rational singularities as well by (c).

In order to show the second claim, let $\widetilde D \in \mZ \grD$ be such that $\mQ D \cap \mZ \grD = \mZ \widetilde D$. Then the inclusions $\mZ D \subset \mZ \widetilde D \subset \mZ \grD$ define a torus $S \doteq \Hom(\mZ \grD / \mZ \widetilde D,\mC^*)$ and a finite group $\grG \doteq \Hom(\mZ \widetilde D / \mZ D,\mC^*)$. Moreover, we have natural actions of $S$ on $C(X)$ and of $\grG$ on $C(X)^S$, and $C_D(X) = (C(X)^S)^\grG$. Therefore, by (b) it follows that $C_D(X)$ has rational singularities as well.
\end{proof}

\vskip 1cm

\end{document}